\documentclass[12pt]{amsart}
\usepackage{graphicx}
\usepackage{amsgen,amsmath,amsxtra,amssymb,amsfonts,amsthm,color,amssymb}
\usepackage{tikz} 
\usetikzlibrary{arrows}
\numberwithin{equation}{section}
\newtheorem{example}{\sc Example}
\newtheorem*{dhef}{Definition}
\newtheorem{theo}{Theorem}

\newtheorem{remark}[example]{\sc Remark}
\newtheorem{prop}[example]{\sc Proposition}

\newtheorem{theoremletter}{Theorem}

\newtheorem{lemmaletter}[theoremletter]{Lemma}
\def\be{\begin{equation}}
\def\ee{\end{equation}}
\def\bc{\begin{cases}}
\def\ec{\end{cases}}
\def\bs{\begin{split}}
\def\es{\end{split}}

\newcommand{\na}{\mathbb{N}}
\newcommand{\z}{\textbf{z}}

\newcommand{\re}{\mathbb{R}}
\newcommand{\rn}{\mathbb{R}^{N}}
\newcommand{\dm}{\mathcal{DM}_{\infty}(\Omega)}

\newcommand{\huz}{W^{1,p}_0 (\Omega)}
\newcommand{\lp}[1]{L^{#1}(\Omega)}
\newcommand{\bv}{BV(\Omega)}
\newcommand{\norma}[2]{\|#1\|_{#2}\}}
\newcommand{\wc}{\rightharpoonup}
\newcommand{\Div}{\textrm{div}}

\def\al{\alpha}

\def\de{\delta}
\def\eps{\varepsilon}

\def\vp{\varphi}
\def\un{u_{n}}

\def\dys{\displaystyle}
\long\def\salta#1{\relax}

\def\rife#1{(\ref{#1})}
\def\io{\int_{\Omega}}
\def\bv{BV(\Omega)}
\def\dcal{\mathcal{D}'(\Omega)}
\def\norma#1#2{\|#1\|_{\lower 4pt \hbox{$\scriptstyle #2$}}}

\def\mis{{\rm meas}}
\def\dive{{\rm div}}
\def\D{D}
\def\Div{\textrm{div}}
\def\g{\gamma}

\author[G. da Silva]{Genival da Silva}
\address{Genival da Silva,
Department of Mathematics,
Texas A\&M University San Antonio,
San Antonio, TX
USA}
\email{gdasilva@tamusa.edu}

\title{An equation involving the 1-Laplacian and a singular nonlinearity}

\begin{document}

\begin{abstract}
We prove existence  of solutions to a nonlinear degenerate elliptic equation of the form
$$
\bc
-\Delta_{1} u+ \frac{|D u|}{(1-u)^{\gamma}}=g & \mbox{in $\Omega$,}\\
u=0 \hfill & \mbox{on $\partial\Omega$,}
\ec
$$
in a suitable sense, where $\Omega $ is a bounded open set of $\rn$, $\g>0$ is a fixed parameter, $g\geq 0 $ is a function in some Lebesgue space.
\end{abstract}

\maketitle

\section{Introduction and statement of results}

Let $\Omega$ be an open bounded set of $\rn$. In this work, we consider the problem
\be\label{main}\tag{D}
\bc
-\Div\left(\frac{\D u}{|\D u|}\right) + \frac{|D u|}{(1-u)^{\gamma}}=g & \mbox{in $\Omega$,}\\
u=0 \hfill & \mbox{on $\partial\Omega$,}
\ec
\ee
where $\g>0$ is a fixed parameter and $g$ belongs to $\lp N$, $g \geq 0$.

The natural space to study solutions to \eqref{main} is $\bv$, namely, the space of functions of bounded variation in $\Omega$. Recall that a function $u$ is of bounded variation if $u\in\lp1$ and its weak derivative $Du$ is a Radon measure with finite total variation. Moreover, if $u\in BV(\Omega)$ then then measure $Du$ has a decomposition into absolutely continuous, jump and Cantor parts, with respect to the Lebesgue measure, i.e. $Du=D^{a}u+D^{j}u+D^{c}u$. Functions of bounded variation have a rich theory and many applications in geometric measure theory and in the calculus of variations, see \cite{evans,ambr}.

In order to make sense of a notion of solution to Problem \ref{main}, we need Anzellotti's theory of vector fields having bounded measure as divergence. We now recall some results from \cite{anz}, \cite{merc}. 

Define
\[
\dm = \{\z\in L^{\infty}(\Omega,\re^{N})\,|\, \Div(\z)\text{ is a bounded measure in }\Omega\}.
\]
For $\z\in\dm$ and $u\in BV(\Omega)\cap C(\Omega)\cap\lp\infty$, we define the distribution $(\z,Du):C_{0}^{\infty}(\Omega)\to \re$ by
\[
\langle (\z,Du),\vp\rangle :=-\io u\vp \Div(\z)-\io u\z\cdot D\vp.
\]
The measure $(\z,Du)$ and $|(\z,Du)|$ are absolutely continuous with respect to $|Du|$. Moreover, if $L$ is Lipschitz and $u$ continuous then
\be\label{rnd}
\frac{d(\z,DL(u))}{d|Du|} = \frac{d(\z,Du)}{d|Du|}\quad\text{ $|Du|$-a.e. in $\Omega$} 
\ee
In \cite{anz} a notion of trace in $L^{\infty}(\partial\Omega)$ is defined for $\z\in\dm$. More precisely, if we denote the trace of $\z$ by $[\z,\nu]$, then the following \textit{Green's formula} holds:
\[
\io u\,\Div(\z) + \io (\z,Du) = \int_{\partial \Omega} u[\z,\nu]d\mathcal{H}^{N-1}
\]
In fact, even if $u$ is not continuous, the definition of  $(\z,Du)$ still makes sense if we substitute $u$ by its precise representative $u^{*}$:
\[
\langle (\z,Du),\vp\rangle :=-\io u^{*}\vp \Div(\z)-\io u\z\cdot D\vp.
\]
Moreover, if $D^{j}u=0$ equation \eqref{rnd} is still true, and more importantly, for $u,v\in BV(\Omega)\cap \lp\infty$:
\[
(v\z,Du) = v^{*}(\z,Du)
\]

\medskip 

We define the following notion of solution to problem \rife{main}. 

\begin{dhef}\label{def}\rm
A function $u\in\bv$ is a solution for problem \rife{main} if $0 \leq u<1$ a.e. in $\Omega$, $|D^{j}u|=0$, and there exists a vector field $\z\in \dm$, with $\norma{\z}{\infty}\le 1$, such that 
\be\label{def1}
\bc
-\Div(\z) + \frac{|D u|}{(1-u)^{\gamma}} = g \text{ in }\dcal\\
(\z,Du)=|Du| \text{ as measures in } \Omega\\
u_{|\partial\Omega} = 0 \text{ in the trace sense}
\ec
\ee
\end{dhef}

Our first existence result is the following.

\begin{theo}\label{t1}\sl
If $\g>1$ and $\norma{g}{\lp N}\le 1$, then problem \rife{main} has at least one solution.
\end{theo}

In \cite{bocpet}, the semilinear version of \eqref{main} is studied. More precisely, the authors prove existence and nonexistence results for the following problem:
\be
\bc\label{lin}
-\Delta u+ \frac{|D u|^{2}}{(1-u)^{\gamma}}=g & \mbox{in $\Omega$,}\\
u=0 \hfill & \mbox{on $\partial\Omega$.}
\ec
\ee
They show that if $\gamma\ge 2$, then \eqref{lin} admits a solution. Notice that their result does not require any smallness assumption on $g$, which is expected due to the linearity of the operator $-\Delta$.

In \cite{bbm}, the authors study a related quasilinear version of \eqref{lin}. More precisely, they prove the existence of solutions for the following problem:
\be
\bc\label{lin2}
-\Div(a(x,u,Du))+ b(x,u,Du) =g & \mbox{in $\Omega$,}\\
u=0 \hfill & \mbox{on $\partial\Omega$.}
\ec
\ee
where $a(x,s,p)$ is a Leray--Lions operator and $b(x,s,p)$ satisfies natural growth conditions together with the sign condition
\[
b(x,s,p)s\ge 0.
\]
In \cite{mazon}, the authors study a related but nonsingular problem. More precisely, they prove the existence of solutions for
\be
\bc
-\Div\left(\frac{\D u}{|\D u|}\right) + |D u|=g & \mbox{in $\Omega$,}\\
u=0 \hfill & \mbox{on $\partial\Omega$.}
\ec
\ee
The notion of solution considered there is similar to ours, although the problem is variational. Moreover, no smallness assumption on $g$ is required, which is again expected due to the absence of singular terms. The authors also provide explicit solutions in a ball, suggesting that some form of rigidity should hold when $\norma{g}{\lp N}$ is sufficiently small. This is indeed the case and will be addressed in Theorem~\ref{t3} below.

If we drop the condition on $\gamma$, we are still able to obtain solutions under suitable assumptions on $g$. More precisely, we have

\begin{theo}\label{t2}\sl
If $g=\lambda f$ and $\lambda< \frac{1}{S_{N,1}\norma{f}{\lp N}}$, with $f\in\lp N$ nonnegative, then problem \rife{main} has a solution for every $\gamma\ge 0$.
\end{theo}

In \cite{gia}, the authors study a class of Dirichlet problems involving the $1$-Laplacian together with strongly singular nonlinearities and lower order gradient terms. More precisely, they prove the existence of solutions for the following problem:
\be
\bc
-\Div\left(\frac{\D u}{|\D u|}\right) = \frac{|D u|}{u^{\gamma}}+\frac{g}{u^{\theta}} & \mbox{in $\Omega$,}\\
u=0 \hfill & \mbox{on $\partial\Omega$.}
\ec
\ee
where $0<\gamma\le 1$ and $0<\theta<1$.

The notion of solution considered there is similar to ours. The authors also derive suitable weak formulations and identify conditions under which the two singular lower order terms can be rigorously interpreted in the distributional setting.

Our last result consists in the analysis of Problem \eqref{main} when $g$ is constant and $\gamma=1$. We consider the following problem.

\be\label{main3}\tag{E}
\bc
-\Div\left(\frac{\D u}{|\D u|}\right) + \frac{|D u|}{1-u}=\lambda  & \mbox{in $\Omega$,}\\
u=0 \hfill & \mbox{on $\partial\Omega$.}
\ec
\ee
Recall that if $\Omega\subset\re^{N}$ is a bounded connected Lipschitz domain. Then by Poincar\'e inequality, there exists a constant $C=C(\Omega)>0$ such that for every $u\in BV_{0}(\Omega)$,
\[
\|u\|_{L^1(\Omega)}
\le
\frac1{h(\Omega)}\,|Du|(\Omega),
\]
where
\[
h(\Omega):=
\inf_{E\subset\Omega,\ |E|>0}
\frac{P(E)}{|E|}.
\]
The positive constant $h(\Omega)$ is called \textit{Cheeger} constant of $\Omega$. Notice that $h(\Omega)$ is well-defined. 

Indeed, since $\Omega$ is bounded, choosing $E=\Omega$ gives
\[
h(\Omega)\le \frac{P(\Omega)}{|\Omega|}<\infty.
\]
To show positivity, let $E\subset\Omega$ be measurable with finite perimeter.
By the isoperimetric inequality,
\[
P(E)\ge N\omega_N^{1/N}|E|^{\frac{N-1}{N}},
\]
and therefore
\[
\frac{P(E)}{|E|}
\ge
N\omega_N^{1/N}|\Omega|^{-1/N}.
\]

Taking the infimum over all admissible sets $E\subset\Omega$ yields
\[
0<
N\omega_N^{1/N}|\Omega|^{-1/N}
\le
h(\Omega)
\le
\frac{P(\Omega)}{|\Omega|}
<\infty.
\]

Therefore the Cheeger constant is well defined.
\begin{theo}\label{t3}\sl
If $\lambda \le h(\Omega)$, then problem \rife{main3} admits only the trivial solution $u\equiv 0$. On the other hand, if $\lambda > h(\Omega)$ and $\Omega=B_{1}$ is the unit ball, then nonconstant solutions exist.
\end{theo}
\begin{remark}
Compare the theorem above with Example~4.1 of \cite{mazon}, which involves a similar but nonsingular equation. Remarkably, the presence of the singularity does not prevent the existence of solutions in the result above.
\end{remark}
\begin{remark}
Notice that in dimension one, $B_{1}=(-1,1)$ and $h(B_{1})=1$. In this case, the solution is given by
\be
u(x)=
\bc
1-e^{1-\lambda },\;  |x|\le \frac{1}{\lambda},\\
1-e^{\lambda (|x|-1)},\; \frac{1}{\lambda}<|x|<1,
\ec
\ee
and the associated vector field is
\be
\z=
\bc
-\lambda x,\;  |x|\le \frac{1}{\lambda},\\
-\frac{x}{|x|},\; \frac{1}{\lambda}<|x|<1.
\ec
\ee

\begin{figure}[h]
\begin{center}
 \includegraphics[scale=0.5]{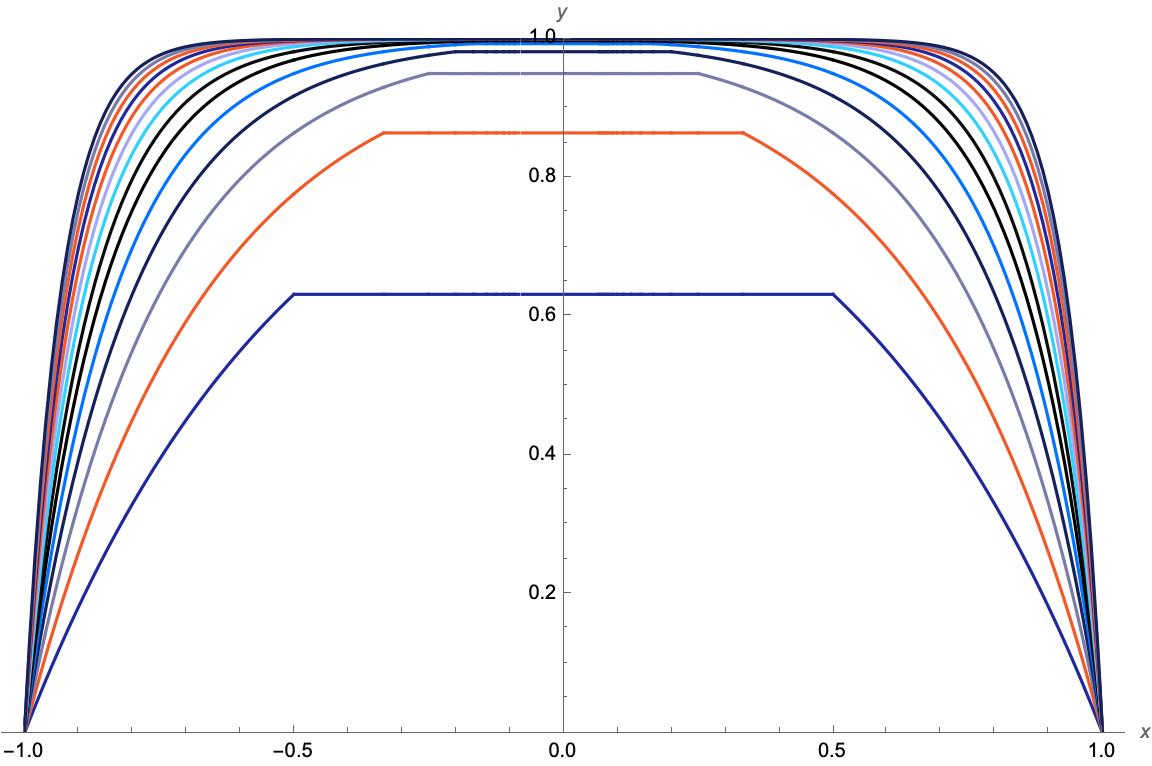}   
 \caption{The graph of $u(x)$ for $\lambda=2,3,4,\ldots,20$.}
\end{center}
\end{figure}

\end{remark}
\medskip

The plan of the paper is as follows. In Section~2 we will prove Theorem \ref{t1} by 
approximating \rife{main} with the equivalent $p$-Laplacian problems. In Section~3 we prove Theorem \ref{t2} and Theorem \ref{t3}.

\medskip

{\bf Notation.} 
\medskip

The following truncation functions will be used throughout the text.
\[
T_k (s)= \max ( -k, \min ( s, k)),\quad
G_k (s)=s-T_k (s)
\]
The Sobolev constant will be denoted by $S_{N,p}$. Recall that for $u\in\huz$ we have $\norma{u}{\lp {p^{*}}}\le S_{N,p} \norma{Du}{p}$, where $p^{*}=\frac{Np}{N-p}$.  We denote $p_{*}$ the Holder conjugate of $p^{*}$.

Throughout the paper, the letter $C$ denotes a positive constant that may vary from line to line.

\section{Proof of Theorem \ref{t1}}
The idea of the proof of Theorem~\ref{t1} is to consider an approximating $p$-Laplacian problem and derive uniform estimates in order to identify a candidate solution as $p\to 1^{+}$. More precisely, we consider
\be\label{pmain}\tag{$D_{p}$}
\bc
-\Div\left(|Du|^{p-2}\D u \right) + \frac{|D u|^{p}}{(1-u)^{\gamma}}=g & \mbox{in $\Omega$,}\\
u=0 \hfill & \mbox{on $\partial\Omega$.}
\ec
\ee

In this setting, the notion of solution is the standard weak one.

\begin{dhef}\label{def2}\rm
A function $u\in\huz$ is said to be a solution of problem \eqref{pmain} if $0 \leq u<1$ a.e. in $\Omega$, $\frac{|\D u|^p}{(1-u)^{\gamma}}\in L^{1}(\Omega)$, and
\[
\io |D u |^{p-2}D u \cdot D \vp +\io \frac{|\D u|^p}{(1-u)^{\gamma}} \vp = \io g \vp,
\]
for every $\vp\in \huz\cap\lp{\infty}$.
\end{dhef}

The proof consists of two main steps: first, proving existence for the approximating $p$-Laplacian problem, and second, showing that the corresponding solutions converge to a solution of the original problem as $p\to 1^{+}$.
\medskip

The following lemma will be used throughout the discussion below. For the proof, see Lemma~6.2 in \cite{boc13}.
\begin{lemmaletter}\label{lemA}
Let $f\in \lp 1$, $g(k)=\io |G_{k}(f)|$ and $A_{k}=\{ |f|>k \}$. Suppose
\[
g(k)\le \beta |A_{k}|^{\al}
\]
for some $\alpha>1$ and $\beta>0$. Then $f\in \lp \infty$ and 
\[
\| f\|_{\lp \infty}\le C\beta 
\]
for some $C=C(\alpha,\Omega)$. 
\end{lemmaletter}

{\bf Step 1. The p-Laplacian approximation} 
\medskip 

We prove the existence of solutions, in the sense of the definition above, to the following problem:

\be\label{pmain2}\tag{P}
\bc
-\Div\left(|Du|^{p-2}\D u \right) + \frac{|D u|^{p}}{(1-u)^{\gamma}}=g & \mbox{in $\Omega$,}\\
u=0 \hfill & \mbox{on $\partial\Omega$,}
\ec
\ee

\begin{theo}\label{tp1}\sl
If $\gamma\ge p$ then problem \eqref{pmain2} has a solution $u\in\huz$.
\end{theo}
\begin{proof}
Fix $n\in\na$, and let $u_{n}\in \huz$ be a weak solution to the approximating problem
\be\label{app1}\tag{$P_{n}$}
-\Div\left(|Du_{n}|^{p-2}\D u_{n} \right) + h_{n}(u_{n})|D u_{n}|^{p}=g_{n}
\ee
where $g_n =T_n (g)$ and $h_n$ is defined by
\be \label{hn}
h_n (s)= \begin{cases}
\hfill 0 \hfill & \mbox{if $s < 0$,} \\
\hfill \frac{ns}{(1-s)^{\gamma}+\frac1n} \hfill & \mbox{if $0 \leq s < \frac1n$,} \\
\hfill \frac{1}{(1-s)^{\gamma}+\frac1n} \hfill & \mbox{if $\frac1n\le s<1$,} \\
\hfill n \hfill & \mbox{if $s\ge 1$.}
\end{cases}
\ee

Notice that $h_{n}$ is bounded, continuous, and satisfies
\[
h_{n}(s)\,s \geq 0
\qquad \forall s\in \re.
\]
By the classical result of \cite{bbm}, there exists a solution $u_{n}\in\huz$ to problem \rife{app1}. However, it is not immediately clear whether $u_{n}\in L^\infty(\Omega)$. In fact, \cite{bbm} states that, in general, the corresponding solutions may be unbounded.

Despite this, the explicit structure of the $p$-Laplacian allows us to obtain boundedness. Since $g_{n}\in L^\infty(\Omega)$ and $h_{n}(s)\,s\ge0$, taking $G_{k}(u_{n})$ as a test function in \rife{app1} yields an estimate for $\int_\Omega |G_{k}(u_n)|$ in the spirit of Lemma~\ref{lemA}. More precisely, we have
\be
\io|DG_{k}(u_{n})|^{p}\le \left(\int_{A_{k}} |g_{n}|^{p_{*}}\right)^{\frac1{p_{*}}}\left(\io |G_{k}(u_{n})|^{p^{*}}\right)^{\frac1{p^{*}}},
\ee
where $A_{k}=\{|u_{n}|>k\}$. Using the Sobolev inequality, we deduce
\be
\left(\io|G_{k}(u_{n})|^{p^{*}}\right)^{\frac{p-1}{p^{*}}}\le S_{N,p}\| g\|_{\lp N} |A_{k}|^{\frac{N-p_{*}}{p_{*}N}},
\ee
where $S_{N,p}$ denotes the Sobolev constant for the embedding $\huz\subset \lp{p^{*}}$.

Finally, applying Hölder's inequality to the left-hand side of the inequality above, we obtain
\be\label{gkb}
\io|G_{k}(u_{n})|\le S_{N,p}^{\frac1{p-1}}\| g\|_{\lp N}^{\frac1{p-1}} |A_{k}|^{\frac{1}{p-1}\left(\frac{p}{p_{*}}-\frac{1}{N}\right)}.
\ee
Since
\[
\frac{1}{p-1}\left(\frac{p}{p_{*}}-\frac{1}{N}\right)
=
\frac{1}{p-1}\left(p-1+\frac{p}{N}-\frac1N\right)
=
1+\frac1N,
\]
it follows from Lemma~\ref{lemA} that $u_{n}\in L^\infty(\Omega)$. 

Taking $-u_{n}^{-}$ as a test function in the weak formulation of \eqref{app1}, we obtain
\[
\io |Du_{n}^{-}|^{p}\le 0.
\]
Hence, $u_{n}$ is nonnegative.

We now claim that $(u_n)$ is uniformly bounded in $\huz$. Indeed, taking $\vp =T_1 (u_n)$ as a test function, we obtain
\[
\io |D\un|^{p-2} \D u_n\cdot \D T_1 (u_n)
+\io h_n(u_n)|\D u_n|^p\, T_1 (u_n)
=
\io T_1 (u_n) g_n .
\]

Since $g_n\leq g$, and by definition $T_1(u_{n}) =1$ and $h_n (u_n)=n$ on the set $\{u_{n} \geq 1\}$, it follows that
\[
\int_{\{u_n\leq 1\}}| \D u_n |^p
+
n \int_{\{u_n\geq 1\}} |\D u_n|^p
\leq
\|g\|_{\lp1}.
\]
In particular,
\be\label{pbound}
 \norma{\un}{\huz}\le \norma{g}{\lp1}^{1/p}.
\ee

Therefore, there exists $u\in\huz$ such that, up to a subsequence, $u_n$ converges weakly to $u$ in $\huz$.
\medskip

\begin{prop}\label{p1}
The convergence $\un\to u$ is strong in $\huz$. Moreover, $0\le \un\le 1$ a.e. in $\Omega$.
\end{prop}
\begin{proof}

Let $\eps>0$, $0\leq k<1$, and choose $\vp = {\frac{1}{\eps}}T_\eps (G_k (u_n))$ as test function. We get
$$
\begin{array}{l}
\dys
{\frac{1}{\eps}} \io |\D u_n|^{p-2} \D u_n\cdot \D T_\eps (G_k (u_n) )
\\
\dys
\quad
+ {\frac{1}{\eps}}\io h_n(u_n)|\D u_n|^p T_\eps (G_k (u_n))
= \io {\frac{1}{\eps}}T_\eps (G_k (u_n)) g_n\,.
\end{array}
$$
Notice that $T_{\eps}(G_{k}(u_{n})) = \eps$ when $u_{n} \geq k+\eps$, thus
\be\label{uke}
\dys \int_{\{ u_n \geq k+\eps\} } h_n(u_n)|\D u_n|^p \leq \int_{\{u_n \geq k\}} g_{n} \leq \int_{\{u_n \geq k\}} g \,.
\ee
By taking the limit as $\eps$ tends to $0^{+}$ and choosing $k=0$, we deduce that
\be\label{hndn}
\io h_n(u_n)|\D u_n|^p \leq \norma{g}{\lp1}\,,
\ee
Since
$$
-\dive(|D\un|^{p-2}\D u_{n}) \,,
$$
is bounded in $\lp1$, it follows from \cite{BM} that (up to subsequences) $\D u_{n}$ converges to $\D u$ almost everywhere in $\Omega$.

\medskip

Now, since $h_{n}$ is increasing, we deduce from \eqref{uke} that 
\be\label{1}
\dys \int_{\{u_n \geq k\} }|\D u_n|^p \leq \frac{1}{h_n(k)}\,\io g\,.
\ee
Setting  $k=1$ we obtain
\be\label{verysmall}
\dys \int_{\{ u_n \geq 1\} }|\D u_n|^p \leq \frac{1}{ n}\, \io g\,. 
\ee
Letting $n$ tend to infinity, and using Fatou lemma together with the almost everywhere convergence of $\D u_{n}$, we have
$$
\int_{\{ u > 1\} }|\D u|^p = 0,
$$
which implies $0 \leq u \leq 1$ almost everywhere in $\Omega$.

Now we prove that $\un\to u $ strongly in $\huz$. Since \[\un=T_{k}(\un)+G_{k}(\un),\] it suffices to show that $T_{k}(\un)$ and $G_{k}(\un)$ converge strongly to $T_{k}(u)$ and  $G_{k}(u)$ respectively. 

We first prove  $T_{k}(\un)\to T_{k}(u)$ strongly in $\huz$.

Fix $0<k<1$, $\eta>0$ and let 
\[
h(s)=\frac{1}{(1-s)^{\gamma}} 
\]

Choose  $\vp = \vp_{ \eta } (T_k (u_n)-T_k (u))$ as test function to obtain

\be\label{trunc1}
\begin{array}{l}
\dys
C_{p}\io |DT_k (u_n)-DT_k (u)|^p \vp_{\eta}'\le\io( |D\un|^{p-2} D u _n -|D u|^{p-2} D u)\cdot \D (T_k (u_n)-T_k (u)) \vp_{\eta}'\\
\dys
\quad
=- \io h_n(u_n)|\D u_n|^p \vp_{\eta} -\io |D u|^{p-2} D u\cdot \D (T_k (u_n)-T_k (u)) \vp_{\eta}+ \io\vp_{\eta} g_{n}\,.
\end{array}
\ee

Notice that $h_n(u_n)\le h_{n}(k)\le h(k)$ on the set ${\{u_n\leq k\}}$. Hence
\be\label{trunc2}
\begin{array}{l}
\dys
\int_{\{u_n\leq k\}} h_n(u_n)|\D u_n|^{p-2}D\un\cdot D\un \vp_{\eta} + \int_{\{u_n\leq k\}} |D u|^{p-2} D u\cdot \D (T_k (u_n)-T_k (u)) \vp_{\eta}\\
\dys
\quad
\ge -h(k) \int_{\{u_n\leq k\}}( |D\un|^{p-2} D u _n -|D u|^{p-2} D u)\cdot \D (T_k (u_n)-T_k (u)) \vp_{\eta} + o(1)\\
\ge -C_{p}h(k) \int_{\{u_n\leq k\}} |DT_k (u_n)-DT_k (u)|^p \vp_{\eta} + o(1)
\end{array}
\ee

Since $ \io\vp_{\eta} g_{n}=o(1)$ as $n\to \infty$, we can combine \eqref{trunc1} and \eqref{trunc2} to obtain:
\[
\io |DT_k (u_n)-DT_k (u)|^p (\vp_{\eta}'-h(k) \vp_{\eta})\le o(1)
\]
The function $\vp_{\eta}$ has the following property:
\be\label{vpla}
\vp_{\eta}'(s) - h (k) \vp_{\eta}(s) \geq 1\,,
\quad
\text{ if }\eta >\frac{ h^2(k)}{8}\,,
\ 
\forall s \in \re
\,.
\ee
Hence, we fix $\eta >\frac{ h^2(k)}{4}$, we obtain 
\[
\io |DT_k (u_n)-DT_k (u)|^p \le o(1)
\]
as desired. Therefore, $T_{k}(\un)\to T_{k}(u)$ strongly in $\huz$.

Finally, to finish the proof we show that $G_{k}(\un)\to G_{k}(u)$ strongly. Recall, from \eqref{1} we have
\be\label{1}
\dys \int_{E} |\D G_{k}(u_n)|^p \leq \frac{1}{h_n(k)}\,\io g\,.
\ee
for any measurable set $E\subseteq \Omega$. So if we choose $n$ large enough $\frac{1}{h_n(k)}<C$, and it follows that the family $\{ D G_{k}(u_n) \}$ is equiintegrable. By Vitali's convergence theorem, $G_{k}(\un)\to G_{k}(u)$ strongly in $\huz$.
\end{proof}
\begin{prop}\label{p2}
The function $u$ satisfies $0\le u<1$ almost everywhere. Moreover, $h_{n}(\un)|D\un|^{p}\to \frac{|Du|^{p}}{(1-u)^{\gamma}}$ strongly in $\lp1$.
\end{prop}
\begin{proof}
Since we already proved that $0\le u\le 1$, it suffices to show that the set $\{u=1\}$ has measure zero. Set 
\[\Phi_n(s)=\int_0^s \sqrt[p]{h_n(t) } dt .\]

From the identity \eqref{hndn} we obtain
$$
\io |\D \Phi_n (u_n) |^p
\leq
\|g\|_{\lp1}\,.
$$
By Poincar\'e inequality we then deduce, for every $\tau > 0$, there exists $C > 0$,
$$
C \int_{\{1-\tau \leq u_n\leq 1+\tau \} } | \Phi_{n} (u_n) |^p
\leq
C \io | \Phi_{n} (u_n) |^p
\leq \|g\|_{\lp1}\,.
$$
Now, since $\Phi_{n}$ is increasing, we deduce that
$$
\mis (\{1-\tau \leq u_n\leq 1+\tau \}) \leq
\frac{\|g\|_{\lp1}}{C| \Phi_n (1-\tau) |^p}
\,.
$$
If $n$ is large enough, then $h_{n}(s) = h(s)$ on $[0,1-\tau]$, so that
\be\label{phip}
\mis (\{1-\tau \leq u_n\leq 1+\tau  \}) \leq
\frac{\|g\|_{\lp1}}{C| \Phi_{p} (1-\tau) |^p}
\,,
\ee
where $\Phi_{p}(t) = \int_{0}^{t}\,\sqrt[p]{\frac{1}{(1-s)^{\gamma}}}\,ds$. Since $\Phi_{p}$ is unbounded on $[0,1)$ because $\gamma\ge p$, it follows that $\{u=1\}$ has measure zero.

We now prove the strong convergence of $h_{n}(\un)|D\un|^{p}$. First notice that since the map
\[
(s,z)\mapsto \frac{|z|^{p}}{(1-s)^{\gamma}}
\]
is continuous, we can easily see that $h_{n}(\un)|D\un|^{p}\to \frac{|Du|^{p}}{(1-u)^{\gamma}}$ almost everywhere. We claim $\{h_{n}(\un)|D\un|^{p}\}$ is equiintegrable.

Let $E\subset \Omega$ be a measurable set and for every $0<k<1$ we have, since $h_{n}$ is increasing, and $h_{n}(s) \leq h(s)$ if $s > \frac 1n$, 
$$
\begin{array}{r@{\hspace{2pt}}c@{\hspace{2pt}}l}
\dys 
\int_E 
h_n (u_n)|\D u_n|^p 
& = &
\dys
\int_{E\cap\{u_n \leq k\}} h_n (u_n)|\D u_n|^p +
\int_{E\cap\{u_n \geq k\}} h_n (u_n)|\D u_n|^p
\\
& \leq &
\dys 
h(k)\int_{E} |\D T_k (u_n) |^p
+ 
\int_{\{u_n \geq k\}} h_n (u_n)|\D u_n|^p \,.
\end{array}
$$
Using \rife{hndn}, we have that 
$$
\int_{\{u_n \geq k\}} h_n (u_n)|\D u_n|^p
\leq
\int_{\{k\leq u_n \leq 1 \}} g
+ \int_{\{ u_n > 1 \}} g\,,
$$
and the last integral goes to zero since $0\le u <1$ almost everywhere in $\Omega$.

Let now $\eps>0$; since $g\in\lp1$, there exists $\de_{\eps} > 0$ such that
$$
\mis(E) < \de_{\eps} \ \Rightarrow\ \int_{E}\,g < \frac{\eps}{2}\,.
$$
Choose $k_{\eps}<1$ such that 
$$
\int_{\{k_{\eps} \leq u_{n} \leq 1\}}\,g < \frac{\eps}{2}\,,
\quad
\forall n \geq n_{0}\,.
$$
Once $k_{\eps} < 1$ is fixed, we already know that $T_{k_{\eps}} (u_n)\to T_{k_{\eps}} u$ strongly in $\huz$; therefore, we can choose $\mis(E)$ small enough so that 
$$
h (k_{\eps})\int_{E} |\D T_{k_{\eps}} (u_n) |^p \leq \frac{\eps}{2}\,,
$$
uniformly with respect to $n$. The conclusion follows by Vitali convergence theorem.

\end{proof} 

We can now finish the proof of Theorem \ref{tp1}. By propositions \ref{p1} and \ref{p2}, we can pass the limit on \eqref{app1}. Thus, $u$ constructed above is a solution.
\end{proof}

Henceforth, the solution constructed above will be denoted by $u_{p}$ in order to emphasize its dependence on $p$.

{\bf Step 2. The limit solution when $p\to 1$} 
\medskip 

Recall that, for $p>1$, estimate \eqref{pbound} yields
\be\label{unip}
 \norma{u_{p}}{\huz}\le \norma{g}{\lp1}^{\frac1p}.
\ee
Hence, the family $\{u_{p}\}$ is uniformly bounded in $W_0^{1,p}(\Omega)$. In particular, by the compactness properties of $BV$ spaces, there exist a function $u\in BV(\Omega)$ and a subsequence, still denoted by $\{u_p\}$, such that
\[
u_{p}\to u \quad \text{in }L^1(\Omega)\text{ and a.e. in }\Omega,
\]
as $p\to 1^{+}$.

Also, from \eqref{gkb} and the assumption $\norma{g}{\lp N}\le 1$, we deduce that
\be\label{cx1}
\io|G_{k}(u_{p})|\le S_{N,p}^{\frac{1}{p-1}} |A_{k}|^{1+\frac{1}{N}}.
\ee
Using the explicit expression for $S_{N,p}$ (see, for instance, \cite{talenti}), we have
\[
S_{N,p}\to \frac{1}{N \omega_{N}^{\frac{1}{N}}}<1
\qquad \text{as } p\to 1^{+}.
\]
Therefore, there exists $p_{0}>1$ such that
\[
S_{N,p}<1
\qquad \forall\, 1<p<p_{0}.
\]
Since
\[
S_{N,p}^{\frac{1}{p-1}}\to 0
\qquad \text{as } p\to1^{+},
\]
it follows that, possibly reducing $p_{0}$,
\[
S_{N,p}^{\frac{1}{p-1}}<1
\qquad \forall\, 1<p<p_{0}.
\]
Hence, by Lemma~\ref{lemA},
\[
\norma{u_{p}}{\lp\infty}\le C,
\]
where the constant $C>0$ is independent of $p$. Since $u_{p}\to u$ almost everywhere, we conclude that $u\in\lp\infty$.

Finally, since $0\le u_{p}<1$ almost everywhere in $\Omega$, we have
\[
0\le u\le 1
\qquad \text{a.e. in }\Omega.
\]
Moreover, by \eqref{phip} and the assumption $\gamma>1$, we obtain
\[
u<1
\qquad \text{a.e. in }\Omega.
\]

{\bf Step 3. Existence of a vector field $\z\in\dm$ satisfying $\norma{\z}{\infty}\le 1$.}
\medskip 

We now prove the existence of a vector field satisfying Definition~\ref{def}. The idea is to extract a convergent subsequence from the family $\{|Du_{p}|^{p-2}Du_{p}\}_{p>1}$. To achieve this, we use Vitali's theorem once again.

First, notice that $\{|Du_{p}|^{p-2}Du_{p}\}_{p>1}$ is equiintegrable. Indeed, for any measurable set $E\subset \Omega$, by using \eqref{unip} we obtain
\be\label{equii1}
\left|\int_{E} |Du_{p}|^{p-2}Du_{p} \right|\le \int_{E} |Du_{p}|^{p-1}\le \left(\int_{E} |Du_{p}|^{p}\right)^{\frac{p-1}{p}}|E|^{\frac1p}\le C |E|^{\frac1p}
\ee
The fact that $\{|Du_{p}|^{p-2}Du_{p}\}_{p>1}$  is uniformly bounded is again a direct consequence of \eqref{unip}. Therefore, by Vitali's theorem, there exists $\z\in L^{1}(\Omega,\rn)$ such that, up to ``subsequence'',
\be
|Du_{p}|^{p-2}Du_{p}\wc \z\text{ as }p\to 1^{+}
\ee
Moreover, by the lower semi-continuity of the total variation combined with Young's inequality we obtain for a given $0\le \vp\in C_{0}^{\infty}(\Omega)$:
\be
\begin{split}
\io \vp \frac{|Du|}{(1-u)^{\gamma}}\le \liminf \io\vp \frac{|Du_{p}|}{(1-u_{p})^{\gamma}}&\le \liminf \left(\frac{1}{p} \io\vp \frac{|Du_{p}|^{p}}{(1-u_{p})^{\gamma }}+\frac{1}{p'} \io\frac{\vp }{(1-u_{p})^{\gamma}}\right)\\
&\le \liminf \io\vp\frac{ |Du_{p}|^{p}}{(1-u_{p})^{\gamma}}
\end{split}
\ee
Since $u_{p}$ satisfies
\be
\io |Du_{p}|^{p-2}Du_{p}D\vp + \io\vp\frac{ |Du_{p}|^{p}}{(1-u_{p})^{\gamma}} =\io g\vp,
\ee
we conclude that if we take the limit $p\to 1^{+}$ we obtain
\be\label{zvp}
\io \z D\vp + \io\vp\frac{ |Du|}{(1-u)^{\gamma}} \le \io g\vp.
\ee
Hence, $\Div(\z)$ is a Radon measure in $\Omega$. Additionally, we can argue similar to equation \eqref{equii1}, and obtain for $1\le q \le p'$:
\be
\left(\io ||Du_{p}|^{p-2}Du_{p}|^{q} \right)^{\frac1q}\le\left(\io |Du_{p}|^{p}\right)^{\frac{p-1}{p}}|\Omega|^{\frac1q-\frac{p-1}{p}}\le C |\Omega|^{\frac1q-\frac{p-1}{p}}.
\ee
It follows that $\norma{z}{\lp q}\le |\Omega|^{\frac1q}$, and letting $q\to \infty$ we deduce that $\norma{z}{\lp \infty}\le 1$.

We claim that $\Div(\z)$ has bounded total variation. Indeed, recall that $\{-\Div(|Du_{p}|^{p-2}Du_{p})\}_{p>1}$ is uniformly bounded in $\lp 1$. Therefore, up to subsequence, it converges weakly-* to some measure, which must be $-\Div(\z)$ by construction. So $-\Div(\z)$ is a Radon measure with finite total variation, thus $\z\in\dm$.

{\bf Step 4. The equation $-\Div(\z) + \frac{|D u|}{(1-u)^{\gamma}} = g$ holds in $\dcal$, and $\z$ satisfies $(\z,Du)=|Du|$.}
\medskip 

Let $0\le \vp\in C_{0}^{\infty}(\Omega)$, and take $e^{-\Phi_{p}(u_{p})}\vp$ as a test function in the weak formulation of problem \eqref{pmain2}. We obtain
\be
\io e^{-\Phi_{p}(u_{p})} |Du_{p}|^{p-2}Du_{p}\cdot D\vp  =\io g\, e^{-\Phi_{p}(u_{p})}\vp.
\ee
Passing to the limit as $p\to 1^{+}$, we obtain
\be
\io e^{-\Phi(u)}\z\cdot D\vp  =\io g\, e^{-\Phi(u)} \vp,
\ee
where
\[
\Phi(s)=\int_{0}^{s}\frac{1}{(1-t)^{\gamma}}\, dt.
\]
Equivalently,
\be
-\Div(e^{-\Phi(u)}\z) = e^{-\Phi(u)}g
\qquad \text{in }\mathcal D'(\Omega).
\ee

Furthermore, by \eqref{zvp} and the fact that $\norma{\z}{\infty}\le 1$, we have
\be
\begin{split}
-\Div(e^{-\Phi(u)}\z)
&=
-e^{-\Phi(u)}\Div(\z)-(\z,D(e^{-\Phi(u)}))\\
&\le
e^{-\Phi(u)}g-e^{-\Phi(u)}|D\Phi(u)|-(\z,D(e^{-\Phi(u)}))\\
&\le
e^{-\Phi(u)}g-e^{-\Phi(u)}|D\Phi(u)|+|D(e^{-\Phi(u)})|\\
&=
e^{-\Phi(u)}g.
\end{split}
\ee
Hence, all the inequalities above are equalities. In particular,
\be\label{eqz}
-\Div(\z)+|D\Phi(u)|=g
\qquad \text{as measures}.
\ee
Moreover, equality in the second inequality yields
\be
-(\z,D(e^{-\Phi(u)}))=|D(e^{-\Phi(u)})|.
\ee

Now let $0\le \vp\in C_{0}^{\infty}(\Omega)$, and take $u_{p}\vp$ as a test function in the weak formulation of problem \eqref{pmain2}. We obtain
\be
\io u_{p} |Du_{p}|^{p-2}Du_{p}\cdot D\vp
+\io \vp |Du_{p}|^{p}
+\io \frac{u_{p}\vp|Du_{p}|^{p}}{(1-u_{p})^{\gamma}}
=
\io g u_{p}\vp.
\ee
Passing to the limit as $p\to 1^{+}$ and using the lower semicontinuity of the $BV$ norm, we deduce
\be
\io u \z\cdot D\vp
+\io \vp |Du|
+\io \frac{u\vp|Du|}{(1-u)^{\gamma}}
\le
\io g u\vp.
\ee
Finally, using \eqref{eqz}, we obtain
\be
\io u \z\cdot D\vp
+\io \vp |Du|
+\io \frac{u\vp|Du|}{(1-u)^{\gamma}}
\le
\io \frac{u\vp|Du|}{(1-u)^{\gamma}}
-\io u\vp \Div(\z).
\ee
Therefore,
\be
\io \vp |Du|
\le
\langle (\z,Du),\vp\rangle.
\ee
Since $\norma{\z}{\infty}\le 1$, it follows that
\be
(\z,Du)=|Du|
\qquad \text{as measures}.
\ee

\begin{prop}
The jump part of $Du$ vanishes, i.e. $D^{j}u=0$.
\end{prop}

\begin{proof}
Let $\Gamma_{k}$ be a sequence of Lipschitz hypersurfaces such that
\be
\mathcal{H}^{N-1}\left(S_{u}\setminus \bigcup_{k=1}^{\infty} \Gamma_{k}\right)=0.
\ee
To prove the proposition, it suffices to show that
\be
|D^{j}u|(\Gamma_{k})=0
\qquad \forall k\in\na.
\ee
Equivalently, we show that for every fixed $k\in \na$ and every $a\in \Gamma_{k}$, there exists an open set $U\ni a$ such that
\[
|D^{j}u|(\Gamma_{k}\cap U)=0.
\]

Take $U$ to be a smooth neighborhood of $a$ satisfying
\[
\overline{U\cap \Gamma_{k}}\subset \Omega,
\]
choose $n_{0}$ such that
\[
\frac{1}{n_{0}}<d(\overline{U\cap \Gamma_{k}},\partial\Omega),
\]
and define, for each $n\ge n_{0}$,
\[
U_{n}:=\left\{x+t\nu(x)\,:\, x\in U\cap \Gamma_{k},\ |t|<\frac1n\right\},
\]
where $\nu(x)$ denotes the unit normal vector to $\Gamma_k$ at $x$. Notice that
\[
U_{n+1}\subset U_{n},
\qquad
\bigcap_{n\ge n_{0}} U_{n}=U\cap \Gamma_{k}.
\]

Since
\[
-\Div(\z)+|D\Phi(u)|=g
\qquad \text{as measures},
\]
we obtain
\be
-\int_{U_{n}} u^* \,\Div(\z)+\int_{U_{n}} u^* |D\Phi(u)|=\int_{U_{n}} u g.
\ee
Using Green's formula, we deduce
\be\label{gf1}
\int_{U_{n}}  (\z,Du)
-\int_{\partial U_{n}}u[\z,\nu]\,d\mathcal{H}^{N-1}
+\int_{U_{n}} u^* |D\Phi(u)|
=
\int_{U_{n}} u g.
\ee

Consider the decomposition
\be
\partial U_{n} = \partial U_{n}^{+}\cup \partial U_{n}^{-}\cup  \partial U_{n}^{0},
\ee
where
\be
\begin{split}
\partial U_{n}^{+} &= \left\{x+\frac1n\nu(x)\,:\, x\in U\cap \Gamma_{k}\right\},\\
\partial U_{n}^{-} &= \left\{x-\frac1n\nu(x)\,:\, x\in U\cap \Gamma_{k}\right\},\\
\partial U_{n}^{0} &= \partial U_{n}\setminus \left(\partial U_{n}^{+} \cup \partial U_{n}^{-}\right).
\end{split}
\ee

Notice that $\partial U_{n}^{0}$ is the lateral boundary of the cylinder $U_{n}$. Hence,
\be
\bigcap_{n\ge n_{0}} \partial U_{n}^{0} = \partial (U\cap \Gamma_{k}),
\ee
and
\be \label{up1}
\lim_{n\to\infty} \int_{\partial U_{n}^{0}}u[\z,\nu]\,d\mathcal{H}^{N-1}=0,
\ee
since
\[
\mathcal{H}^{N-1}(\partial (U\cap \Gamma_{k}))=0.
\]

Moreover,
\be \label{up2}
\begin{split}
\lim_{n\to\infty} \int_{\partial U_{n}^{+}}u[\z,\nu]\,d\mathcal{H}^{N-1}
&=
\int_{U\cap \Gamma_{k}}u_{\Gamma_{k}}^{+}[\z,\nu]\,d\mathcal{H}^{N-1},\\
\lim_{n\to\infty} \int_{\partial U_{n}^{-}}u[\z,\nu]\,d\mathcal{H}^{N-1}
&=
\int_{U\cap \Gamma_{k}}u_{\Gamma_{k}}^{-}[\z,-\nu]\,d\mathcal{H}^{N-1}.
\end{split}
\ee

Combining \eqref{up1} and \eqref{up2}, we obtain
\be
\lim_{n\to\infty} \int_{\partial U_{n}}u[\z,\nu]\,d\mathcal{H}^{N-1}
=
\int_{U\cap \Gamma_{k}}(u_{\Gamma_{k}}^{+}-u_{\Gamma_{k}}^{-})[\z,\nu]\,d\mathcal{H}^{N-1}.
\ee

Passing to the limit in \eqref{gf1}, we obtain
\be \label{up3}
\int_{U\cap \Gamma_{k}} (\z,D^{j}u)
-\int_{U\cap \Gamma_{k}}(u_{\Gamma_{k}}^{+}-u_{\Gamma_{k}}^{-})[\z,\nu]\,d\mathcal{H}^{N-1}
+\int_{U\cap \Gamma_{k}} \frac{u^*|D^{j}u|}{(1-u^*)^{\gamma}}
=0.
\ee

Now observe that
\be
\left|
\int_{U\cap \Gamma_{k}}
(u_{\Gamma_{k}}^{+}-u_{\Gamma_{k}}^{-})[\z,\nu]
\,d\mathcal{H}^{N-1}
\right|
\le
\int_{U\cap \Gamma_{k}}
|u_{\Gamma_{k}}^{+}-u_{\Gamma_{k}}^{-}|
\,d\mathcal{H}^{N-1}
=
|D^{j}u|(U\cap \Gamma_{k}).
\ee

Since $(\z,Du)=|Du|$ as measures, we also have
\[
(\z,D^{j}u)=|D^{j}u|.
\]
Therefore, \eqref{up3} yields
\be
\int_{U\cap \Gamma_{k}}
\frac{u^*}{(1-u^*)^{\gamma}}
\,|D^{j}u|
\le 0.
\ee

Finally, on the jump set,
\[
u^*=\frac{u^++u^-}{2},
\]
and since $0\le u<1$ a.e., we have
\[
0\le u^*<1
\qquad |D^{j}u|\text{-a.e.}
\]
Hence,
\[
\frac{u^*}{(1-u^*)^\gamma}\ge0,
\]
which implies
\[
|D^{j}u|(U\cap \Gamma_k)=0.
\]
The conclusion follows.
\end{proof}
Finally, we prove that $u$ has zero trace.

\begin{prop}
$u=0$ on $\partial\Omega$ in the trace sense.
\end{prop}

\begin{proof}
Taking $u_{p}$ as a test function in the weak formulation of problem \eqref{pmain2}, we obtain
\be
\io |Du_{p}|^{p}
+\io \frac{u_{p} |Du_{p}|^{p}}{(1-u_{p})^{\gamma}}
=
\io g u_{p}.
\ee
Passing to the limit as $p\to 1^{+}$ and using the lower semicontinuity of the $BV$ norm, we deduce
\be\label{trace1}
\io |Du|
+\int_{\partial \Omega} |u|\,d\mathcal{H}^{N-1}
+\io \frac{u^* |Du|}{(1-u^*)^{\gamma}}
+\frac12\int_{\partial \Omega} u^{2}\,d\mathcal{H}^{N-1}
\le
\io g u.
\ee

On the other hand, using Green's formula together with
\[
-\Div(\z)+\frac{|Du|}{(1-u^*)^\gamma}=g,
\]
we obtain
\be
\begin{split}
\io g u
&=
-\int_\Omega u^*\,\Div(\z)
+\io \frac{u^*|Du|}{(1-u^*)^\gamma}\\
&=
\langle (\z,Du),1\rangle
-\int_{\partial\Omega} [\z,\nu]u\,d\mathcal H^{N-1}
+\io \frac{u^*|Du|}{(1-u^*)^\gamma}.
\end{split}
\ee
Since $(\z,Du)=|Du|$ and $\|\z\|_\infty\le1$, combining the previous identity with \eqref{trace1} yields
\be
\int_{\partial \Omega} |u|\,d\mathcal{H}^{N-1}
+\frac12\int_{\partial \Omega} u^{2}\,d\mathcal{H}^{N-1}
\le
-\int_{\partial \Omega} [\z,\nu]u\,d\mathcal H^{N-1}.
\ee
Using the estimate
\[
|[\z,\nu]|\le1,
\]
we infer that
\be
\int_{\partial \Omega} |u|\,d\mathcal{H}^{N-1}
+\frac12\int_{\partial \Omega} u^{2}\,d\mathcal{H}^{N-1}
\le
\int_{\partial\Omega}|u|\,d\mathcal H^{N-1}.
\ee
Therefore,
\[
\frac12\int_{\partial \Omega} u^{2}\,d\mathcal{H}^{N-1}\le0,
\]
and consequently
\[
u=0
\qquad \mathcal H^{N-1}\text{-a.e. on }\partial\Omega.
\]
\end{proof}

\section{Proof of Theorems \ref{t2} and \ref{t3}}
\begin{proof}[{\sl Proof of Theorem \ref{t2}}]
We can apply the argument of the theorem \ref{t1} \textit{mutatis mutandis}, except for a few modifications, which we now discuss.

First, equation \eqref{cx1} becomes
\be\label{cx2}
\io|G_{k}(u_{p})|\le S_{N,p}^{\frac{1}{p-1}}\norma{\lambda f}{\lp N}^{\frac{1}{p-1}} |A_{k}|^{1+\frac{1}{N}}.
\ee
thus 
\be\label{cx3}
\norma{u_{p}}{\lp \infty}\le C_{0} S_{N,p}^{\frac{1}{p-1}}\norma{\lambda f}{\lp N}^{\frac{1}{p-1}}
\ee
After using the hypothesis and taking the limit $p\to 1^{+}$, we have
\be
\norma{u}{\lp \infty}< 1
\ee
Therefore, the left of identity \eqref{phip} still goes to zero when $\tau\to 0$ and the condition $\gamma\ge p$ does not have to be used. The rest of the argument remains the same.
\end{proof}

\begin{proof}[{\sl Proof of Theorem \ref{t3}}]
Let $u$ be a solution and let $\z$ be the corresponding vector field as in definition \eqref{def1}. Testing the first equation with $u$, we obtain
\[
\io |Du|
+
\io \frac{u}{1-u}\,|Du|
=
\lambda\io u
\]
Hence
\[
\io \frac{|Du|}{1-u}
=
\lambda\io u
\]

As before, set $v:=-\log(1-u)$. By Poincar\'e inequality:
\[
h(\Omega)\io v(x)\,dx
\le \io |Dv|
=
\lambda\io u
\]

Since $v\ge u$, we have
\[
(h(\Omega)-\lambda)\io u
\le 
0
\]
Therefore, if $\lambda<h(\Omega)$ only the trivial solution $u\equiv 0$ with  $z=-\lambda x$ is possible. Moreover, if $\lambda=h(\Omega)$ then $\io v=\io u$, which is possible only if $u\equiv 0$ again.

Since $h(B_{1})=N$, the proof of the second statement in Theorem~\ref{t3} can be obtained if we set
\be
u(x)=
\bc
1-\left(\frac{\lambda}{N}\right)^{N-1}e^{N-\lambda },\; x\in B_{\frac{N}{\lambda}},\\
1-|x|^{-(N-1)}e^{\lambda (|x|-1)},\; x\in B_{1}\setminus B_{\frac{N}{\lambda}}.
\ec
\ee
and 
\be
\z =
\bc
-\lambda \frac{x}{N},\; x\in B_{\frac{N}{\lambda}}a,\\
-\frac{x}{|x|},\; x\in B_{1}\setminus B_{\frac{N}{\lambda}}.
\ec
\ee
\end{proof}
\begin{remark}
An interesting question that remains open in this work is whether the solution obtained in Theorem~\ref{t1} is always constant. More generally, one may ask under which conditions nonconstant solutions exist.
\end{remark}
\begin{remark}
It is clear that the argument used in the proof of Theorem~\ref{t3} fails when $\gamma\neq 1$. This naturally raises the following question: does one still have rigidity results similar to those in Theorem~\ref{t3} when $\gamma\neq 1$?
\end{remark}

\end{document}